\documentclass[12pt,a4paper]{amsart}
\usepackage{mathrsfs}
\usepackage{amsfonts}
\usepackage[active]{srcltx}
\usepackage{enumerate}

\usepackage{mathrsfs}
\usepackage[active]{srcltx}
\usepackage{enumerate}
\usepackage{amsmath,amssymb,xspace,amsthm}

\usepackage{graphicx}
\usepackage{amscd}
\usepackage{amsfonts}
\usepackage{color}
\usepackage[mathscr]{eucal}
\usepackage{latexsym}

\theoremstyle{definition}
\newtheorem{thm}{Theorem}
\newtheorem{lem}[thm]{Lemma}

\newtheorem{cor}[thm]{Corollary}

\newtheorem{rem}[thm]{Remark}
\newtheorem{exa}[thm]{Example}

\numberwithin{equation}{section}
\newtheorem{claim}{Claim}

\def\Vir{\operatorname{Vir}}

\newcommand{\C}{\mathbb C}

\newcommand{\N}{\mathbb{N}}
\newcommand{\Z}{\mathbb{Z}}

\renewcommand{\a}{\alpha}

\renewcommand{\phi}{\varphi}

\def\tp{\widetilde{\phi}}

\def\b{\beta}

\begin{document}
\title[Tensor product weight modules]{Tensor product weight modules over the Virasoro algebra}
\author{Hongjia Chen, Xiangqian Guo and Kaiming Zhao}


\begin{abstract} The tensor product
of highest weight modules with intermediate series modules over the
Virasoro algebra was discussed by Zhang  \cite{Z} in 1997. Since
then the irreducibility problem for the tensor products has been
open. In this paper, we determine the necessary and sufficient
conditions for these tensor products to be simple. From non-simple
tensor products, we can get other interesting simple Virasoro
modules.
We also obtain that any two such  tensor products are isomorphic if
and only if the corresponding highest weight modules and
intermediate series  modules are isomorphic respectively.
Our method is to develop a ``shifting technique" and to widely use
Feigin-Fuchs' Theorem on singular vectors of Verma modules over the
Virasoro algebra.\vskip5pt

\vskip5pt

\end{abstract}
\maketitle



\section{Introduction}

\vskip 5pt We denote by $\mathbb{Z}$, $\mathbb{Z}_+$, $\mathbb{N}$
and $\mathbb{C}$
   the sets of  all integers, nonnegative integers, positive integers and complex numbers,
   respectively.

The Virasoro algebra $\Vir$ is an infinite dimensional Lie algebra
over the complex numbers $\C$, with basis $\{d_n,C\,\, |\,\, n \in
\Z\}$ and defining relations
$$
[d_{m},d_{n}]=(n-m)d_{n+m}+\delta_{n,-m}\frac{m^{3}-m}{12}C, \quad
m,n \in \Z,
$$
$$
[C, d_{m}]=0, \quad m \in \Z,
$$ which is the universal central extension
of the so-called infinite dimensional Witt algebra. 
The algebra $\Vir$ is one of the most important Lie algebras both in
mathematics and in mathematical physics, see for example \cite{IK,
KR} and references therein. In particular,
 it has been widely used in   quantum physics \cite{GO},
conformal field theory \cite{DMS}, Kac-Moody algebras \cite{K,MoP}, vertex
operator algebras \cite{DMZ,FZ}, and so on.

The representation theory on the Virasoro algebra has been
attracting a lot of attentions from mathematicians and physicists.
There are two classical families of simple Harish-Chandra
$\Vir$-modules: highest (lowest) weight modules (completely
described in \cite{FF}) and the so-called intermediate series
modules. In \cite{M} it is shown that these two families exhaust all
simple   Harish-Chandra modules. In \cite{MZ1} it is even shown that
the above modules exhaust all simple weight modules admitting a
nonzero finite dimensional weight space. Now mathematicians have
shifted their attentions to non-Harish-Chandra simple modules,
mainly simple weight modules with infinitely dimensional weight
spaces and non-weight modules.


As for simple weight modules over $\Vir$ with infinitely dimensional
weight spaces, the earliest examples were constructed from tensor
products of simple highest weight modules and intermediate series
modules (\cite{Z}) in 1997. Later a  $4$-parameter family of simple
weight modules with infinitely dimensional weight spaces were given
in \cite{CM} in 2001. Recently in \cite{LLZ}   another huge class of
simple modules with infinitely dimensional weight spaces were
obtained using non-weight modules introduced in \cite{MZ2}.

At the same time for the last decade, various families of nonweight
simple Virasoro modules were studied in \cite{FJK, GLZ, LGZ, LZ, Y,
MW, OW}. These (except the modules in \cite{LZ} and \cite{MW}) are
basically various versions of Whittaker modules constructed using
different tricks. In particular, all the above Whittaker modules and
even more were described in a uniform way in \cite{MZ2}.

%

In the paper \cite{Z}, Zhang considered the tensor products of
simple highest weight modules with simple intermediate series
modules, and he provided some sufficient conditions for the tensor
products to be simple. Since then the irreducibility problem for the
tensor products has been open. It was even not known in \cite{Z}
whether there was a non-simple tensor product when the highest
weight module is not a Verma module.  Our purpose in the present
paper is to completely solve this problem. More precisely we
determine the necessary and sufficient conditions for such tensor
products to be simple, using two polynomials obtained from the
singular vectors of the original highest weight modules. Moreover we
can determine the conditions for two of these tensor product modules
to be isomorphic. We remark that the tensor products of intermediate
series modules over the Virasoro algebra never gives irreducible
modules \cite{Zk}.

To describe our results, we recall some notations for the Virasoro
algebra and its modules.
We first define the modules of intermediate series
$V_{\alpha,\beta}$ for any $\a, \b\in\C$. As a vector space
$V_{\a,\b}=\oplus_{n\in\Z}\C v_n$ and the action of $\Vir$ on
$V_{\a,\b}$ is given by
\begin{equation}\label{intermediate}
d_m \cdot v_n = (\a+n +m\b)v_{m+n}, \quad C \cdot v_n =0,\ \forall
m,n\in\Z. \end{equation} It is well known that $V_{\a,\b} \cong
V_{\a +n,\b}$, and $V_{\a,0} \cong V_{\a,1}$ if $\a \notin \Z$. 
It is also well known that $V_{\a,\b}$  is simple if and only if $\a
\notin\Z$ or $\b\notin\{0,1\}$. Moreover $V_{0,0}$ has a unique
nonzero proper submodule $\C v_0$ and we denote
$V^\prime_{0,0}=V_{0,0}/\C v_0$; $V_{0,1}$ has a unique nonzero
proper submodule $V^\prime_{0,1}=\sum_{i\neq0}\C v_i$ and $V^\prime_{0,0}\cong V^\prime_{0,1}$. 
 For convenience we will denote $V^\prime_{\a,\b}=V_{\a,\b}$ whenever
$V_{\a,\b}$ is simple. We remark that $V^\prime_{\a,0}\cong
V^\prime_{\a,1}$ for all $\a\in\C$.

To avoid repetition, throughout this paper when we write
$V'_{\a,\b}$ \textbf{we always assume that $0 \leq \mathfrak{Re} \a
<1$ and $\b\ne 1$,} where $\mathfrak{Re}\a$ is the real part of
$\a$.

We recall that a Virasoro module $V$ is a {\bf weight module} if it
is the sum of all its weight spaces $V_{\lambda}=\{v\in V\,|\,
d_0v=\lambda v\}$ for some $\lambda\in\C$. Vectors in $V_\lambda$
are call {\bf weight vectors}. An element $x$ in the universal
enveloping algebra $U(\Vir)$ is called {\bf homogeneous} if $[d_0,
x]=mx$ for some $m\in\Z$ and $m$ is called the {\bf degree} of $x$,
denoted by $\deg(x)=m$. Let $U(\Vir)_{m}$ be the subspace consisting
of all elements of degree $m$.

Next we define the highest weight modules. Denote
$\Vir_{\pm}=\sum_{i\in\N}\C d_{\pm i}$. For any $h, c\in\C$, we let
$\C u$ be the $1$-dimensional module over the subalgebra
$\Vir_+\oplus\C d_0\oplus\C C$ defined by
\begin{equation}\label{highest}
\Vir_+ u=0,\ d_0 u=hu, \ {\rm and\ }\quad C u=c u. \end{equation}
Then we get the induced $\Vir$-module, called \textbf{Verma module}:
$$M(c,h)=U(\Vir)\bigotimes_{U(\Vir_+\oplus\C d_0\oplus\C C)}\C u.$$
 Any nonzero quotient module of
$M(c,h)$ is called a \textbf{highest weight module} with highest
weight $(c,h)$.

It is well known that the Verma module $M(c,h)$ has a unique maximal
submodule $J(c,h)$ and the corresponding simple quotient module is
denoted by $V(c,h)$. A nonzero weight vector $u'\in M(c,h)$ is
called a \textbf{singular vector} if $\Vir_+ u'=0$. It is clear that
$J(c,h)$ is generated by all singular vectors in $M(c,h)$ not
contained in $\C u$, and that $M(c,h)=V(c,h)$ if and only if
$M(c,h)$ does not contain any other singular vectors besides $\C u$.
In \cite{FF} (or in \cite{A} which is a refined form)  the singular
vectors are described explicitly. In particular, $J(c,h)$ can be
generated by at most two singular vectors. If $J(c,h)$ is generated
by two singular vectors, we can find homogeneous $Q_1, Q_2\in
U(\Vir_-)$ such that $J(c,h)=U(\Vir_-)Q_1u+U(\Vir_-)Q_2u$. If
$J(c,h)$ is generated by one singular vector, we can find the unique
$Q_1\in U(\Vir_-)$ up to a scalar multiple such that
$J(c,h)=U(\Vir_-)Q_1u$; in which case we set $Q_2=Q_1$ for
convenience. When $M(c,h)=V(c,h)$ we set $Q_2=Q_1=0$.


Fix any $c,h,\a,\b\in\C$ with $0 \leq \mathfrak{Re} \a <1$   and
$\b\ne 1$. For any $n\in\Z$, we have a linear map $\phi_n$ from
$U(\Vir_-)$ to $\C$ defined by
$$\phi_n(d_{-k_r}\cdots d_{-k_1})=\prod_{j=1}^r(k_j\b-\a-n-\sum_{i=1}^jk_i),$$
for all $d_{-k_r}\cdots d_{-k_1}\in U(\Vir_-)$. For more details of
this map see Sect.2.

For any    integer $n$ let
$W^{(n)}$ be the submodule of $V(c,h) \otimes V^\prime_{\a,\b}$
generated by $u\otimes v_{i}, i>n$.
Then our main results can be
stated as


\begin{thm}\label{simplicity} Let $c,h,\a,\b\in\C$ with $0 \leq \mathfrak{Re} \a
<1$ and $\b\ne 1$.
\begin{enumerate}[(a)]
\item $V(c,h) \otimes V^\prime_{\a,\b}$ is simple if and only if
there is no (nonzero if $(\a,\b)=(0,0)$) integer $n$ such that
$\phi_n(Q_1)=\phi_n(Q_2)=0$.
\item If $n$ is the maximal integer (nonzero if $(\a,\b)=(0,0)$) with $\phi_n(Q_1)=\phi_n(Q_2)=0$,  then
$W^{(n)}$ is the unique simple submodule and all its nonzero weight spaces are infinite-dimensional.
\end{enumerate}
\end{thm}


\begin{thm}\label{isomorphism}
Let $V$ and $V'$ be any two highest weight modules (not necessarily
simple) of the Virasoro algebra, and let $\a,\b,\a_1,\b_1\in\C$.
Then $V \otimes V^\prime_{\a,\b} \cong V' \otimes
V^\prime_{\a_1,\b_1}$ if and only if $V \cong V'$ and $V'_{\a,\b}
\cong V'_{\a_1,\b_1}$.
\end{thm}

As a direct consequence we have

\begin{thm}\label{isomorphism-coro} Let $c,h,\a,\b, c_1,h_1,\a_1, \b_1\in\C$ with $0 \leq \mathfrak{Re}
\a, \mathfrak{Re} \a_1 <1$ and $\b, \b_1\ne 1$. Then $ V(c,h)
\otimes V^\prime_{\a,\b}\cong V(c_1,h_1) \otimes V^\prime_{\a_1,
\b_1} $ if and only if $c_1=c, h_1=h, \a_1=\a$ and
$\b_1=\b$.\end{thm}

The paper is organized as follows. In Sect.2, we prove Theorem 1.
The main method is to develop a ``shifting technique" and to widely
use Feigin-Fuchs' Theorem on singular vectors of Verma modules over
the Virasoro algebra. Theorem 2 is proved in Sect.3, and we show
that the simple tensor products are not isomorphic to other known
simple weight modules. In the last section we give several examples
to illustrate our results. In particular, when $V(c,h) \otimes
V^\prime_{\a,\b}$ is not simple, we still can get simple modules by
considering submodules or subquotients of it.

\section{Simplicity}\label{section_simlicity}


In this section we will prove Theorem 1. Fix any $c, h, \a, \b\in\C$
with $0 \leq \mathfrak{Re} \a <1$ and $\b\ne 1$. Let $V$ be a
highest weight module (not necessarily simple) with highest weight
vector $u$ of highest weight $(c, h)$ and $V_{\a,\b}$ have a basis
as in \eqref{intermediate}. We will consider the tensor product
modules $V\otimes V_{\a,\b}$ and $V\otimes V'_{\a,\b}$. Let us first
introduce our ``shifting technique" for the tensor product.

\def\ot{\otimes}
The vector space $L=V\otimes \C[t^{\pm1}]$ can be endowed with a
$\Vir$-module structure via
$$d_k (Pu \otimes t^n)=(d_k+\a+n-\deg(P)+k\b)Pu \otimes t^{n+k},$$
for all $k, n\in\Z$ and homogeneous $P\in U(\Vir_-)$, and the action
of $C$ is the scalar $c$. When $\a=\b=0$, we see that $L$ has a
submodule spanned by $Pu\otimes t^{\deg(P)}$ for all homogeneous
$P\in U(\Vir_-)$, which is just a copy of $V(c,h)$; let $L'$ be the
corresponding quotient module in this case, and we denote $L'=L$
otherwise.

It is easy to check that $V\otimes V_{\a,\b}\cong L$ as Virasoro
modules via the following map: for any $m\in\Z_+$
\begin{equation}\label{identify}
V\otimes V_{\a,\b} \rightarrow L, \quad Pu \ot v_n \mapsto Pu\otimes
t^{n+\deg(P)},\ \forall\ n\in\Z, P\in U(\Vir_-)_{-m}.
\end{equation}
Also this map induces an isomorphism $V\otimes V'_{\a,\b}\cong L'$
in the case $\a=\b=0$. Thus in the following of this section we will
consider $L$ and $L'$ instead of $V\otimes V_{\a,\b}$ and $V\otimes
V'_{\a,\b}$. The advantage of the ``shifting technique"  of the
notation for $L=\bigoplus_{n\in\Z} V\ot t^n$ is  the weight space
decomposition, that is,
$$V\ot t^n=\{ x\in L\ |\ d_0x=(\a+h+n)x \},\ \forall\ n\in\Z.$$
We will see the power of this advantage in the following proofs.

 The
following observation is obvious

\begin{lem}  \label{span}
The Virasoro module $L'$  is generated by $\{u \otimes t^k| k \in
\Z\}$. \qed
\end{lem}

From now on in this section, we assume that  $V=V(c,h)$ is the
simple highest weight module with highest weight $(c,h)$. Thus
$L\cong V(c,h)\ot V_{\a,\b}$ and $L'\cong V(c,h)\ot V'_{\a,\b}$. The
following lemma is crucial in this paper.

\begin{lem}
\label{submod} For any nonzero $\Vir$-submodule $W$ of $L'$, there
exists $N \in \Z$ such that $V(c,h) \otimes t^n \subseteq W$ for all
$n \geq N$.
\end{lem}
\begin{proof} Since $L'$ is a weight module, there exist subspaces
$W_n\subseteq V(c,h)$ for all $n\in\Z$ such that
$W=\bigoplus_{n\in\Z} W_n\ot t^n$. For any nonzero vector $w \in
W_n$ we can find homogeneous $P_i\in U(\Vir_-)$ such that
$$
w = \sum_{i=1}^r P_{i}u,\quad \text{where\ } 0 \geq \deg(P_1)>
\deg(P_2)
> \cdots > \deg(P_r).
$$ Choose $w$ with $r$ being minimal among all nonzero
elements in all $W_n, n\in\Z$. Suppose this $w\in W_k$. Denote
$l_i=-\deg(P_i).$ If $\a=\b=0$  we further assume that
$\deg(P_i)\neq k$, i.e., $l_i\neq -k$ for all $i=1,\cdots, r$ since
we are working within $L'$. For $m,n
> -\deg(P_r)$, we have
\begin{equation*}
d_md_n(w \otimes t^k) =\sum_{i=1}^r (\a+n+k+l_i+\b m)(\a+k+l_i+\b n) w_i \otimes t^{m+n+k}
\end{equation*}
\begin{equation*}
d_{m+n}(w \otimes t^k) =\sum_{i=1}^r (\a+k+l_i+\b(m+n)) w_i \otimes
t^{m+n+k}.
\end{equation*}
Thus for all $j: 1\le j\le r$ we have
\begin{equation*}
\begin{split}
&(\a+k+l_j+\b(m+n))d_md_n(w \otimes t^k)\\
&\hskip 1cm -(\a+n+k+l_j+\b m)(\a+k+l_j+\b n)d_{m+n}(w \otimes t^k)  \\
 =& \sum_{i=1}^r \Big((\a+k+l_j+\b(m+n))(\a+n+k+l_i+\b m)(\a+k+l_i+\b n)  \\
&\quad - (\a+n+k+l_j+\b m)(\a+k+l_j+\b n) (\a+k+l_i+\b(m+n))\Big)\\
&\hskip 1cm \cdot w_i \otimes t^{m+n+k}  \\
 =& \sum_{i=1}^r (l_i-l_j)\Big(\b^2(m^2+mn+n^2)+\b mn+ \b (2\a+2k+l_i+l_j)(m+n) \\
&\hskip2cm\quad +(\a+k+l_i)(\a+k+l_j)\Big)w_i \otimes t^{m+n+k}  \in
W.
\end{split}
\end{equation*}
By the minimality of $r$, we have
$$
\b^2(m^2+mn+n^2)+\b mn+ \b (2\a+2k+l_i+l_j)(m+n)$$ $$
+(\a+k+l_i)(\a+k+l_j)=0,
$$
for all $i \neq j$ and $m,n > l_r$. If $r>1$, then we can deduce
$\b=\a=0$ and
$(k+l_i)(k+l_j)=0$, impossible. 
As a result, we have $r=1$ and $w=P_1u \in W_k$ such that
$\deg(P_1)\neq k$ when $\a=\b=0$.

There exists $N_1 \in \Z$ such that $d_nP_1u=0$ and
$\a+k-\deg(P_1)+\b n \neq 0$ for any $n \geq N_1-k$. Thus we have
\begin{equation*}
d_n \cdot (P_1u \otimes t^{k})=(\a+k-\deg(P_1)+\b n) (P_1u \otimes
t^{n+k}) \in W
\end{equation*}
and hence $P_1u \in W_m$ for all $m \geq N_1$.

For any $i\in\N$ and $m\geq N_1$, we have $P_1u\in W_{m+i}$ and
\begin{equation*}
d_i \cdot (P_1u \otimes t^{m})=(d_i+\a+m-\deg(P_1)+\b i)(P_1u
\otimes t^{m+i}) \in W,
\end{equation*}
yielding  that $d_iP_1u\in W_{m+i}$ for all $m\geq N_1$.
Inductively, we can show $xP_1u \in W_{m+\deg(x)}$ for all
homogeneous $x \in U(\Vir_+)$ and $m \geq N_1$. Since $V(c,h)$ is an
irreducible $\Vir$-module, we can find some homogeneous $y\in
U(\Vir_+)$ such that $yP_1u=u$ and hence $u\in W_{m}$ for all $m\geq
N=N_1+\deg(y)$.

For any $i\in\N$ and $m\geq N$, we have 
\begin{equation*}
d_{-i} \cdot (u \otimes t^{m+i})=(d_{-i}+\a+m+i-\b i)(u \otimes
t^{m}) \in W
\end{equation*}
which implies that $d_{-i}u\in W_m$ for all $m\geq N$. Proceeding by
downward induction on $\deg(P), P\in U(\Vir_-)$ we can deduce that
$Pu\in W_m$ for all homogeneous $P\in U(\Vir_-)$ and $m\geq N$, that
is, $W_m=V(c,h)$ for all $m \geq N$, as desired.
\end{proof}

From the proof of the above lemma we can get the following
corollaries.

\begin{cor}\label{coro1} For any $n\in\N$, the submodule generated by all $V(c,h)\ot t^{k}, k>
n$ is the same as the submodule generated by all $u\ot t^k, k>n$. We
denote this submodule by $W^{(n)}$.
\end{cor}

\begin{cor}\label{coro2}
$L'$ is simple if and only if it is generated by $u \otimes t^k$ for
some sufficiently large 
$k \in \N$.
\end{cor}


Now we introduce our linear map  $\tp_n: M(c,h) \rightarrow \C$ for
$n\in\Z$. Denote by $T(\Vir_-)$ the tensor algebra of $\Vir_-$,
i.e.,
\begin{equation*}
T(\Vir_-) = \C \oplus (\Vir_-) \oplus (\Vir_- \otimes \Vir_-) \oplus \cdots ({\Vir_-}^{\otimes k}) \oplus \cdots
\end{equation*}
with multiplication being the tensor product, which is a free
associative algebra. For any $k \in \N$, define $\deg(d_{-k}) = -k$
and $\deg(1) = 0$, then $T(\Vir_-)$ becomes a $\Z_-$-graded algebra.
Now for any $n \in \Z$ and $\a,\b \in \C$, we can inductively define
a linear map $\phi_n: T(\Vir_-) \rightarrow \C$ as follows:
\begin{equation}\label{def_phi_n}
\quad \phi_n(1)=1,\quad \phi_n(d_{-k}P) = -(\a+n+k-\deg(P)-k \b)
\phi_n(P),
\end{equation}
for any homogeneous element $P \in T(\Vir_-)$. Notice that
$U(\Vir_-)$ is just the quotient $T(\Vir_-)/J$ where $J$ is the
two-sided ideal generated by
$d_{-i}d_{-j}-d_{-j}d_{-i}-(i-j)d_{-i-j}$ for $i,j \in \N$. Now we
will show $J \subseteq \ker(\varphi_n)$. By the definition of
$\phi_n$, it is enough to show
\begin{equation*}
\phi_n((d_{-i}d_{-j}-d_{-j}d_{-i}-(i-j)d_{-i-j})P) =0
\end{equation*}
for any homogeneous element $P \in T(\Vir_-)$. Let $m=\deg(P)$, we
have
\begin{equation*}
\begin{split}
&\phi_n\Big((d_{-i}d_{-j}-d_{-j}d_{-i})P\Big)\\
=& (n-m+j+i+\a-i \b)(n-m+j+\a-j \b)\phi_n(P) \\
& \quad -(n-m+i+j+\a-j \b)(n-m+i+\a-i \b)\phi_n(P) \\
= & (j-i)(n-m+i+j+\a-(i+j)\b)\phi_n(P)  \\
= & (i-j) \phi_n(d_{-i-j}P).
\end{split}
\end{equation*}
Thus $\phi_n$ induces a linear map on $U(\Vir_-)$, still denoted by
$\phi_n$. Since the Verma module $M(c,h)$ is free of rank $1$ as a
$U(\Vir_-)$-module, we can define the linear map
\begin{equation*}
\tp_n: M(c,h) \rightarrow \C, \quad \tp_n(P u) = \phi_n(P), \
\forall\ P \in U(\Vir_-).
\end{equation*}

\begin{lem} \label{mod_W}
Let $W$ be a $\Vir$-submodule of $L'$ with $W\supseteq W^{(n)}$ for
some $n\in\Z$. Then  $$Pu \otimes t^n \equiv \phi_n(P)u \otimes t^n\
\ (\hskip-.4cm \mod W), \ \forall\ P \in U(\Vir_-).$$
\end{lem}
\begin{proof} Let $Q \in U(\Vir_-)$ be a homogeneous element and suppose that the result holds for this $Q$.
Then for any $k \in \N$, we have
\begin{equation*}
d_{-k}\Big(Qu \otimes
t^{n+k}\Big)=\Big((d_{-k}+\a+n+k-\deg(Q)-k\b)Qu\Big) \otimes t^n\in
W,
\end{equation*}
which indicates
\begin{equation*}
\begin{split}
(d_{-k}Qu) \otimes t^n &\equiv -(\a+n+k-\deg(Q)-k\b)(Qu) \otimes t^n  \\
&\equiv  -(\a+n+k-\deg(Q)-k\b)\phi_n(Q)u \otimes t^n  \\
&\equiv \phi_n(d_{-k}Q)u\otimes t^n\ \ (\hskip-.4cm \mod W).
\end{split}
\end{equation*}
The lemma follows from induction on $\deg(Q)$ and linearity of
$\phi_n$.
\end{proof}

Let $J(c,h)$ be the maximal submodule of $M(c,h)$. By  \cite{FF}, we
can assume that $J(c,h)$ is generated by two singular vectors $Q_1$
and $Q_2$ for some homogeneous $Q_1, Q_2\in U(\Vir_-)$. We have made
the assumption that $Q_1=Q_2$ if $J(c,h)$ can be generated by a
singular vector and $Q_1=Q_2=0$ if $M(c,h)$ itself is simple. Under
this assumption, $Q_1$ and $Q_2$ are both homogeneous and unique up
to nonzero scalars.
%
%
%
%
%
Now we can give the proof of Theorem \ref{simplicity}.

\noindent {\it \textbf{Proof of Theorem 1.}} Recall that  $\b\ne 1$.
Let $W$ be a nonzero $\Vir$-submodule of $L'\cong V(c,h) \otimes
V^\prime_{\a,\b}$. By Lemma \ref{submod} we know that there is $N
\in \Z$ such that $V(c,h) \otimes t^k \subseteq W$ for all $k>N$.
By Lemma \ref{mod_W} we have
\begin{equation}
Q_iu \otimes t^N- \phi_N(Q_i)u \otimes t^N \in W,\quad i=1, 2.
\end{equation}

``If part of (a).'' By the hypothesis, 
we have $\phi_N(Q_1) \neq 0$ or $\phi_N(Q_2) \neq 0$, we get $u
\otimes t^N \in W$. Inductively, we get $u \otimes t^k \in W$ for
all $k \in \Z$. By Lemma \ref{span}, we have $W = L'$, that is, $L'$
is simple.

``Only if part of (a).'' Now we assume that the equations
$\phi_n(Q_1)=\phi_n(Q_2)=0$ has at least one solution for $n$
(nonzero if $(\a,\b)=(0,0)$).

We first consider the case that $\phi_n(Q_1)=\phi_n(Q_2)=0$  has
only finitely many such solutions and prove (b). If
$(\a,\b)\ne(0,0)$, let $N_0$ be the largest integer  such that
$\phi_n(Q_1)=\phi_n(Q_2)=0$. Similarly as in the ``if part", we get
that $u\ot t^k\in W$ for all $k>N_0$ by induction. If
$(\a,\b)=(0,0)$, let $N_0$ be the largest \textbf{nonzero} integer
such that $\phi_n(Q_1)=\phi_n(Q_2)=0$. Since $u\ot t^0=0$ is already
in $W\subseteq L'$, again we can get $u\ot t^k\in W$ for all
$k>N_0$. Recall that $W^{(N_0)}$ is the $\Vir$-submodule of $L'$
generated by all $V(c,h) \otimes t^k, k> N_0$. In both cases we
conclude that $W\supseteq W^{(N_0)}$ from Corollary 6, implying that
$W^{(N_0)}$ is the smallest submodule of $L'$ which has to be
simple. Since some weight spaces of $W^{(N_0)}$ are
infinite-dimensional, from \cite{MZ1} we know that any nonzero
weight spaces are infinite-dimensional. Thus (b) follows.

Now we go back to the proof of (a).
Let $N$ be any (nonzero if $(\a,\b)=(0,0)$) integer such that
$\phi_N(Q_1)=\phi_N(Q_2)=0$. We will show that $W^{(N)}$ is a proper
submodule of $L'$. By the definition of $\phi_N$ in
\eqref{def_phi_n}, we obtain that
\begin{equation*}\begin{split}
\tp_N(J(c,h)) & =\tp_N\Big(U(\Vir_-)Q_1u+U(\Vir_-)Q_2u\Big)\\
&=\phi_N\big(U(\Vir_-)Q_1\big) +\phi_N\big(U(\Vir_-)Q_2\big)=0.
\end{split}\end{equation*} Thus $\tp_N$ induces a linear map $V(c,h)
\rightarrow \C$ which sends $P u$ to $\phi_N(P)$ for any $P \in
U(\Vir_-)$.
Since $\tp_N(u)=1$ and hence $\ker(\tp_N)\ne V(c,h)$.


First, by the PBW theorem, the weight space of $W^{({N})}$ with
weight $\a+h+{N}$ is
\begin{equation*}\begin{split}
W^{({N})}_{N} \ot t^{N} & = \sum_{k\in\N} d_{-k} \Big( V(c,h)
\otimes
t^{{N}+k} \Big)\\
= & \sum_{k \in\N}\sum_{P\in U(\Vir_-)}
\Big((d_{-k}+\a+{N}+k-\deg(P)-k\b) Pu\Big) \otimes t^{N},
\end{split}\end{equation*} or in other words,
\begin{equation*}
W^{({N})}_{N}  = \sum_{k \in\N}\sum_{P\in U(\Vir_-)}
\Big((d_{-k}+\a+{N}+k-\deg(P)-k\b) Pu\Big),
\end{equation*}
where the second summation is taken over all homogeneous $P$. 
\\ On the other hand, for any homogeneous element $P$, 
we have
\begin{equation*}\begin{split}
& \tp_{N} \Big((d_{-k}+\a+{N}+k-\deg(P)-k\b)Pu\Big)\\
= & \phi_{N}(d_{-k}P)+(\a+{N}+k-\deg(P)-k\b)\phi_N(P)=0.
\end{split}\end{equation*}
That is $W^{({N})}_{N} \subseteq \ker(\tp_{N})$, as desired. By
Lemma \ref{mod_W} we see that actually $W^{({N})}_{N}
=\ker(\tp_{N})$. Thus $W^{(N)}$ is a proper submodule of $L'$, i.e.,
$L'$ is not simple.\qed
Let $\Phi$ be the set of all (nonzero if $(\a,\b)=(0,0)$)
integers $n$ such that $\phi_n(Q_1)=\phi_n(Q_2)=0$. From the above
theorem, we can have the following comments on the submodules of
$L'$.

\begin{rem}\label{remark}
(1) We have a sequence of submodules of $L'$:
\begin{equation}\label{sequence}
\cdots \subseteq W^{(n+1)}\subseteq W^{(n)} \subseteq W^{(n-1)}
\subseteq \cdots
\end{equation}

If $Q_1=Q_2=0$, that is, $V(c,h)=M(c,h)$, then $\Phi=\Z$ except that
$\Phi=\Z\setminus\{0\}$ when $(\a,\b)=(0,0)$ and all inclusions
$W^{(n+1)} \subseteq W^{(n)}$ are proper for $n\in\Phi$. It is easy
to see that $W^{(n)}/W^{(n-1)}$ is a highest weight module of
highest weight $\a+h+n$ for any $n\in\Phi$. In particular $V(c,h)\ot
V'_{\a,\b}$ is an extension of countably many highest weight
modules.

If $Q_1 \neq 0$ or $Q_2 \neq 0$, then $\Phi$ is a finite subset of
$\Z$ and the only proper inclusions in \eqref{sequence} are
$W^{(n)}\subsetneq W^{(n-1)}$ for $ n\in\Phi$. If we write
$\Phi=\{n_1,n_2,\ldots,n_r\}$ with $n_1 < n_2 <\cdots <n_{r}$ and
take any $n_0<n_1$, then \eqref{sequence} can be simplified as:
\begin{equation}\label{sequence-finite}
0 \subsetneq W^{(n_r)} \subsetneq \cdots \subsetneq W^{(n_{2})}
\subsetneq W^{(n_1)} \subsetneq W^{(n_0)}=L'.
\end{equation}
We have known that $W^{(n_r)}$ is the unique minimal nonzero proper
submodule of $L'$. By replacing $N$ with $n_i, 1\leq i\leq r$ in the
proof of Theorem \ref{simplicity}, we see that
$$W^{(n_{i-1})}_{n_i}=V(c,h)=W^{(n_i)}\oplus (\C u\ot t^{n_i}),$$
and $W^{(n_{i-1})}$ is generated by $W^{(n_i)}$ and $u\ot t^{n_i}$.
This implies that $W^{(n_{i-1})}/W^{(n_i)}$ is a highest weight
module with highest weight $\a+h+n_i$. Then we can obtain that
$W^{(n_r)}$ is a simple weight $\Vir$-module with infinitely
dimensional weight spaces.

(2) Recall that $V_{\a,0}\cong V_{\a,1}$ for $\a\neq 0$ and in this
case they correspond to the same module $L'$ while different linear
map $\phi_n$, to distinguish them we denote by $\phi_n^{(\a,0)}$ and
$\phi_n^{(\a,1)}$ respectively.
Then for any homogeneous element $P \in U(\Vir_-)$, we have
\begin{equation*}
\varphi_n^{(\a,0)}(P) = \frac{\a + n-\deg(P)}{\a+n}
\varphi_n^{(\a,1)}(P),
\end{equation*}
which results in the same $\Phi$ and hence the same sequence of
submodules of $L'$ in \eqref{sequence-finite}.
\end{rem}


\section{Isomorphism} \label{isomorphisms}
In this section we will prove Theorem 2 and compare the tensor
products  $V(c,h) \otimes V^\prime_{\a,\b}$ with other known simple
weight modules.

 \noindent{\it \textbf{Proof of Theorem 2.}}
Fix any complex numbers $c ,h, c_1, h_1, \a, \b, \a_1, \b_1$ with $0
\leq \mathfrak{Re} \a, \mathfrak{Re} \a_1 <1$,
  $\b\ne 1$ and $\b_1\ne 1$. Let $V$ (resp. $V_1$) be a highest weight module
generated by highest weight vector $u$ (resp. $u_1$) with highest
weight $(c,h)$ (resp. $(c_1, h_1)$).
As we did in Section \ref{section_simlicity}, we will identify $V\ot
V_{\a,\b}$ and $V_1\ot V_{\a_1, \b_1}$ with $L=V\ot\C[t^{\pm1}]$ and
$L_1=V_1\ot\C[t^{\pm1}]$ respectively via \eqref{identify} and we
also identify $V\ot V'_{\a,\b}$ and $V_1\ot V'_{\a_1, \b_1}$ with
$L'$ and $L_1'$ respectively.

The ``if part" is trivial. We now
prove the ``only if part". 
It is clear that $c=c_1$.

Now assume $\psi: L '
\longrightarrow
L'_1$ is an isomorphism of $\Vir$-modules. Fix any $k\in\Z$ such
that $k\neq0$ when $(\a,\b)=(0,0)$.
Since $\psi(u \otimes t^k)$ and $u \otimes t^k$ are of the same
weight, so $\psi(u \otimes t^k)=w \otimes t^l$ for some nonzero
vector $w=\sum\limits_{i=0}^{s}w_i$ with $w_i\in\C P_iu_1$,
$\deg(P_i)=-i$ and $\a+h+k=\a_1+h_1+l$. We may assume that
$\a_1+l+i\ne0$ if $\b_1=0$.  For any $m,n\geq s+1$, we have
\begin{equation*}
\begin{split}
& \psi\big(d_md_n(u \otimes t^k)\big) = (\a+k+n+\b m)(\a+k+\b n)
\psi(u \otimes t^{m+n+k}) \\
= & d_md_n(w \otimes t^l)\\ = &\sum_{i=0}^s (\a_1 +l+i+n+\b_1
m)(\a_1+l+i+\b_1 n) w_i \otimes t^{m+n+l}
\end{split}
\end{equation*}
and
\begin{equation*}
\begin{split}
& \psi\big(d_{m+n}(u \otimes t^k)\big)=(\a+k+\b(m+n)) \psi(u \otimes t^{m+n+k}) \\
= & d_{m+n}(w \otimes t^l) =\sum_{i=0}^s (\a_1+l+i+\b_1(m+n)) w_i
\otimes t^{m+n+l}.
\end{split}
\end{equation*}
This implies
\begin{equation*}
\begin{split}
& (\alpha+k+\beta(m+n))(\alpha_1 +l+i+n+\beta_1 m)(\alpha_1+l+i+\beta_1 n)\\
= & (\alpha_1+l+i+\beta_1(m+n))(\alpha+k+n+\beta m)(\alpha+k+\beta
n),
\end{split}
\end{equation*}
for all $m,n\geq s+1$ and $0\leq i\leq s$ with $w_i\neq 0$. Rewrite
it as a polynomial of $m$ and $n$ as follows
\begin{equation*}
\begin{split}
& \b \b_1 (\b_1-\b)mn(m+n) + \b \b_1(\a_1+l+i-\a-k)(m^2+n^2)\\
+ & \Big((\a+k)\b_1(\b_1-2\b-1)+(\a_1+l+i)\b(1+2\b_1-\b)\Big)mn \\
+ & \Big((\a+k)(\a_1+l+i)(\b_1-\b)+\b (\a_1+l+i)^2-\b_1(\a+k)^2\Big)(m+n) \\
+ &(\a+k)(\a_1+l+i)(\a_1+l+i-\a-k)=0,
\end{split}
\end{equation*}
which holds for all $i$ with $w_i \neq 0$ and $m,n > s$. Hence we
get
\begin{enumerate}
\item \label{coeff1} $(\a+k)(\a_1+l+i)(\a_1+l+i-\a-k)=0$,
\item \label{coeff2} $\Big(\b (\a_1+l+i)+\b_1 (\a+k)\Big) (\a_1+l+i-\a-k)=0$,
\item \label{coeff3} $\b \b_1(\a_1+l+i-\a-k) =0$,
\item \label{coeff4} $(\a+k)\b_1(\b_1-1)=(\a_1+l+i)\b(\b-1)$,
\item \label{coeff5} $\b \b_1 (\b_1-\b) =0$.
\end{enumerate}



If $\b=0$ or $\b_1=0$, from the above formulas and the assumption
that $0 \leq \mathfrak{Re} \a, \mathfrak{Re} \a_1 <1$,
  $\b\ne 1$ and $\b_1\ne 1$ we can easily deduce that $\b=\b_1=0$ and $\a=\a_1$. Next we
assume that $\b\b_1\ne0$. Again, from the above formulas and the
assumption that $0 \leq \mathfrak{Re} \a, \mathfrak{Re} \a_1 <1$,
  $\b\ne 1$ and $\b_1\ne 1$ we can easily deduce that $\b=\b_1$ and $\a=\a_1$.
In all cases we can deduce that  $k=l+i$ for all $w_i\neq0$, which
implies that there is only one integer, say $r$ such that
$w_r\neq0$. Now we get $\psi(u \otimes t^k)= P_ru_1 \otimes
t^{k-r}$, where $r = h_1-h$ since $\a+h+k=\a_1+h_1+l$. We will
replace $P_r$ with  $P_{r,k}$ since it also depends on $k$. Note
that
$$\sum\limits_{k\in
\Z}U(\Vir_-)(u \otimes t^k) = V \otimes V'_{\a,\b}\ {\rm for\ }
(\a,\b) \neq (0,0)$$ and $$\sum\limits_{k\in
\Z\setminus\{0\}}U(\Vir_-)(u \otimes t^k) = V \otimes V'_{\a,\b}
{\rm\ for\ } (\a,\b) = (0,0),$$ which together with the fact that
$\psi$ is an isomorphism implies that
$$\sum\limits_{k\in
\Z}U(\Vir_-)(P_{r,k}u_1 \otimes t^{k-r}) = V_1 \otimes
V'_{\a_1,\b_1}\ {\rm for\ } (\a,\b) \neq (0,0)$$ and
$$\sum\limits_{k\in \Z\setminus\{0\}}U(\Vir_-)(P_{r,k}u_1 \otimes
t^{k-r}) = V_1 \otimes V'_{\a_1,\b_1} {\rm\ for\ } (\a,\b) =
(0,0).$$
 The above equalities force that $P_{r,k}u_1$ is a nonzero scalar multiple of
 $u_1$. Thus  $r=0$, that is, $h=h_1$ and there exist $c_k\in\C^*$
 such that
$\psi(u \otimes t^k)=c_k u_1 \otimes t^k$ for all $k\in\Z$ but not
zero if $(\a,\b)=(0,0)$.

For any $n \in \N$, we have
\begin{equation*}
\begin{split}
c_{n+k}(\a+k+\b n) u_1 \otimes t^{n+k} & = \psi((\a+k+\b n) u \otimes t^{n+k}) = \psi(d_n(u \otimes t^k)) \\
&=d_n(\psi(u \otimes t^k)) = c_kd_n(u_1 \otimes t^k) \\
&= c_k(\a+k+\b n)u_1 \otimes t^{n+k}.
\end{split}
\end{equation*}
If $\a+k+\b n \neq 0$, we get $c_{n+k}= c_k$. Now for any (nonzero
if $(\a,\b)=(0,0)$) $k,l \in \Z$, choosing a sufficiently large $n
\in \N$ such that
\begin{equation*}
n+k\neq0,\ \a+k+\b n \neq 0 \text{\ and\ } \a+l+\b(n-l+k) \neq 0,
\end{equation*}
then we have $c_k =c_{n+k}$ and $c_l = c_{(n-l+k)+l}=c_{n+k}=a$ for
some $a \in \C^*$. By replacing $\psi$ with its multiple we may
assume that $\psi(u \otimes t^k)=u_1 \otimes t^k$ for all $k\in\Z$
but not zero if $(\a,\b)=(0,0)$. Suppose $X$ be the maximal subspace
of $U(\Vir_-)$ such that
$$\psi((P u) \otimes t^k) = (P u_1) \otimes t^k, \ \forall\ P \in
X, \ \forall k\in\Z.$$ We can easily see that $X=U(\Vir_-)$ and the
map $V\to V_1$ defined by $Pu\to Pu'$ is a module isomorphism.  So
$V \cong V_1$.\qed

Now we compare the tensor products  $V(c,h) \otimes
V^\prime_{\a,\b}$ with other known simple weight modules. So far
known simple weight Virasoro modules are from \cite{CM, LLZ}. It is
proved  in \cite{LLZ} that $V(c,h) \otimes V^\prime_{\a,\b}$ (or its
minimal submodule) is not isomorphic to any module in \cite{LLZ}.
Next we compare $V(c,h) \otimes V^\prime_{\a,\b}$ with simple weight
modules from \cite{CM}.

\begin{thm} Let $c,h,\a,\b, a, b, \gamma, p\in\C$.
  Then $V(c,h) \otimes V^\prime_{\a,\b}$ (or
its minimal submodule) is not isomorphic to any irreducible module
$E_a(b, \gamma, p)$ defined in \cite{CM}.
\end{thm}

\begin{proof} Let us recall the Casimir operators
$Q_n=d_0^2+nd_0-d_{-n}d_n\in U(\Vir)$ for any $n\in\N$. From Page
174 on \cite{CM}, we know that for any $w\in E_a(b, \gamma, p)$,
$$\dim{\text{span}}\{Q_nw\,\,|\,\,n\in\N\}<\infty.$$
At the same time, in $V(c,h) \otimes V^\prime_{\a,\b}$ (or its
minimal submodule) we have
$$\dim{\text{span}}\{Q_n(u\otimes v_k)\,\,|\,\,n\in\N\}=\infty,$$
for any $k\in \Z$ with $v_k\ne0$. The theorem follows.
\end{proof}


\section{Examples}

In this section, to illustrate our Theorem \ref{simplicity} we will
present some examples for special values of $c,h$ and $\a,\b$.
Notations are as in Section \ref{section_simlicity}. To distinguish
different modules relative to different $\a,\b$ we will add
subscripts for certain modules, such as $W_{\a,\b}^{(n)}$ and so on,
if necessary. Recall that $0\leq \Re(\a)<1$ and
$\b\neq 1$. 

\begin{exa} We first give some examples such that the unique maximal submodule $J(c,h)$ of the corresponding
Verma module $M(c,h)$ can be generated by only one singular vector.
We point out that the expression for $h$ at the very beginning in
Sect.5 in \cite{A} should be $h=\frac{m^2-(p+q)^2}{4pq}$.

(I). Let $(c, h)=(1, 0)$. In this case we take $p=-q=1$ and $m=0$ in
Sect.5 of \cite{A}.  From Theorem A in \cite{A} we get that
$J(1,0)=U(\Vir_-)Qu$,\\ where $Q=d_{-1}$. It is clear that
$\phi_n(Q)=(\b-\a-n-1)$. From Theorem 1 we see that
$V(1,0)\ot V'_{\a,\b}$ is simple if and only if $\b-\a \notin\Z$. 
Moreover, if $\b - \a \in \Z$, then $V(1,0)\ot V'_{\a,\b}$ has a
unique simple submodule $W_{\a,\b}^{(\b-\a-1)}$ and $\left(V(1,0)\ot
V'_{\a,\b}\right)/W_{\a,\b}^{(\b-\a-1)}$ is a highest weight module
of highest weight $(1,\b-1)$. It is easy to see that similar result
holds for any module $V(c,0)$ provided that the corresponding
maximal submodule $J(c,0)$ is generated by $d_{-1}u$ (see Example
14).

(II). Let $(c, h)=(1, -\frac{1}{4})$. In this case we take $p=-q=1$
and $m=1$ in Sect.5 of \cite{A}. Then we see that $J(1,
-\frac{1}{4})$ is generated by only one singular vector of degree
$2$. By direct computation we can deduce that $J(1,
-\frac{1}{4})=U(\Vir_-)Qu$ where $Q=d_{-1}^2+d_{-2}$. We have
$$
\phi_n(Q)=(\b-\a-n-2)(\b-\a-n-1)+(2\b-\a-n-2).
$$
Solving the equation $\phi_n(Q)=0$ for $n$,  we get
\begin{equation*}
n_1=\b-\a-1 + \sqrt{1-\b}, \; n_2=\b-\a-1 - \sqrt{1-\b}.
\end{equation*}

(1). $V(1,-\frac{1}{4})\ot V'_{\a,\b}$ is simple if $n_1,n_2 \notin
\Z$.

(2). If for some $\a,\b$, the numbers $n_1 \in \Z$ (resp. $n_2 \in
\Z$) and $n_2 \notin \Z$ (resp. $n_1 \notin \Z$), then
$V(1,-\frac{1}{4})\ot V'_{\a,\b}$ has a unique simple submodule
$W_{\a,\b}^{(n_1)}$ (resp. $W_{\a,\b}^{(n_2)}$). Moreover,
$\left(V(1,-\frac{1}{4})\ot V'_{\a,\b}\right )/W_{\a,\b}^{(n_1)}$
(resp. $\left(V(1,-\frac{1}{4})\ot V'_{\a,\b}\right
)/W_{\a,\b}^{(n_2)}$) is a highest weight module of highest weight
$(1,\b-\frac{5}{4} + \sqrt{1-\b})$ (resp. $(1,\b-\frac{5}{4} -
\sqrt{1-\b})$).

(3). If $n_1,n_2 \in \Z$, then we have two cases
\begin{enumerate}
\item[(i).] $\a=0$, $\b=1-k^2$, $n_1=-k^2+k$ and $n_2=-k^2-k$ for $k \in \N$;
\item[(ii).] $\a=\frac{1}{4}$, $\b=1-(l+\frac{1}{2})^2$, $n_1=-l^2$ and $n_2=-(l+1)^2$ for $l \in \Z_+$.
\end{enumerate}
In case $\a=0$ and $\b=0$, the module $V(1,-\frac{1}{4})\ot
V'_{0,0}$ has a unique simple submodule $W_{0,0}^{(-2)}$ and
$\left(V(1,-\frac{1}{4})\ot V'_{0,0}\right )/W_{0,0}^{(-2)}$ is a
highest weight module of highest weight $(1,-\frac{9}{4})$. In the
rest cases, $V(1,-\frac{1}{4})\ot V'_{\a,\b}$ has a unique simple
submodule $W_{\a,\b}^{(n_1)}$ and another proper submodule
$W_{\a,\b}^{(n_2)}$ strictly containing $W_{\a,\b}^{(n_1)}$;
moreover, $W_{\a,\b}^{(n_2)}/W_{\a,\b}^{(n_1)}$ is a highest weight
module of highest weight $(1,n_1+\a-\frac{1}{4})$ and
$\left(V(1,-\frac{1}{4})\ot V'_{\a,\b}\right) /W_{\a,\b}^{(n_2)}$ is
a highest weight module of highest weight $(1,n_2+\a-\frac{1}{4})$.

(III). Let $(c, h)=(1, -1)$. In this case we take $p=-q=1$ and
$m=2$. Then from \cite{FF} we see that $J(1, -1)$ is generated by
only one singular vector of degree $3$.  By direct computation we
can deduce that $J(1, -1)=U(\Vir_-)Qu$. where
$Q=d_{-1}^3+4d_{-2}d_{-1}+2d_{-3}$. We deduce that
\begin{equation*}
\begin{split}
\phi_n(Q)= &(\b-\a-n-3)(\b-\a-n-2)(\b-\a-n-1)\\
& \quad +4(2\b-\a-n-3)(\b-\a-n-1) + 2(3\b-\a-n-3) \\
= (\b-& \a-n)(n^2+2(1-\b +\a) n+2 \a -2 \b \a +\b^2-3+\a ^2+2\b ).
\end{split}
\end{equation*}
Here, we just consider the case that all roots of the equation
$\phi_n(Q)=0$ for $n$ are integers. It is not hard to see that we
have two cases:
\begin{enumerate}
\item $\a=0$, $\b=1-k^2$, $n_1=-k^2+2k$, $n_2=-k^2+1$ and $n_3=-k^2-2k$ for $k \in \N$;
\item $\a=\frac{3}{4}$, $\b=1-(l+\frac{1}{2})^2$, $n_1=-l^2+l$, $n_2=-l^2-l$ and $n_3=-l^2-3l-2$ for $l \in \Z_+$.
\end{enumerate}
In case $\a=0$ and $\b=0$, the module $V(1,-1)\ot V'_{0,0}$ has
different proper submodules $W_{0,0}^{(-3)}$ and $W_{0,0}^{(1)}$
where $W_{0,0}^{(1)}$ is simple.  We know that the quotients
$W_{0,0}^{(-3)}/W_{0,0}^{(1)}$ and $\left(V(1,-1)\ot V'_{0,0}\right
)/W_{0,0}^{(-3)}$are highest weight modules. In the remaining cases,
$V(1,-1)\ot V'_{\a,\b}$ has three different proper submodules
$W_{\a,\b}^{(n_i)}$ where $i=1,2,3$.
\end{exa}


\begin{exa} Now we give some examples such that the unique maximal submodule
$J(c,h)$ of the Verma module $M(c,h)$ cannot be  generated by one
singular vector.

(I). Let $(c, h)=(0, 0)$. We have $J(0,
0)=U(\Vir_-)Q_1u+U(\Vir_-)Q_2 u$, where $Q_1=d_{-1}$ and
$Q_2=d_{-2}$. Then $\phi_n(Q_1)=(\b-\a-n-1)$ and
$\phi_n(Q_2)=(2\b-\a-n-2)$. From $\phi_n(Q_1)=\phi_n(Q_2)=0$ we
deduce $\a=0$, $\b=1$. From the assumption that   $\b\ne1$, we do
not have any solutions for $n$ with $\phi_n(Q_1)=\phi_n(Q_2)=0$.
Then $V(0,0)\ot V'_{\a,\b}$ is always simple. Indeed, we have
$V(0,0) \ot V'_{\a,\b} \cong V'_{\a,\b}$.

(II). Let $(c, h)=(-\frac{22}{5}, 0)$. In this case we take $p=2,
q=-5$ and $m=3$.  Then from \cite{FF} we see that $J(c,h)$ is
generated by two singular vectors of degree $1$ and $4$
respectively. It is not hard to verify that
$J(c,h)=U(\Vir_-)Q_1u+U(\Vir_-)Q_2 u$, where $Q_1=d_{-1}$ and
$Q_2=3d_{-2}^2+5d_{-4}$. Then we can get
$$
\phi_n(Q_1)=\b-\a-n-1
$$
and
$$
\phi_n(Q_2)=3(2\b-\a-n-4)(2\b-\a-n-2)+5(4\b-\a-n-4).
$$
Plugging $n=\b-\a-1$ in the second equation, we have
$$
\phi_n(Q_2)=3(\b-3)(\b-1)+5(3\b-3)=3(\b-1)(\b+2).
$$
Since $\b \neq 1$, then $V(-\frac{22}{5}, 0) \ot V'_{\a,\b}$ is
simple if and only if $(\a,\b) \neq (0,-2)$. Moreover,
$V(-\frac{22}{5}, 0) \ot V'_{0,-2}$ has a unique simple submodule
$W_{0,-2}^{(-3)}$, and $  \left(V(-\frac{22}{5}, 0) \ot
V'_{0,-2}\right)/W_{0,-2}^{(-3)}$ is a highest weight module of
highest weight $(-\frac{22}{5},-3)$ which is an irreducible Verma
module.

(III). Let $(c, h)=(\frac{1}{2}, -\frac12)$.  In this case we take
$p=3, q=-4$ and $m=5$.  From \cite{FF} we see that $J(c,h)$ is
generated by two singular vectors of degree $2,3$.  It is not hard
to verify that $J(c,h)=U(\Vir_-)Q_1u+U(\Vir_-)Q_2 u$, where
$Q_1=3d^2_{-1}+4d_{-2}$,  $Q_2=4d_{-1}^3+12d_{-2}d_{-1}+3d_{-3}$,
and $Q_1u$ is a singular vector while $Q_2u$ is obtained by
subtracting some element in $U(\Vir_-)Q_1u$ from the other singular
vector. Then
$$
\phi_n(Q_1)=3(\b-\a-n-2)(\b-\a-n-1)+4(2\b-\a-n-2), $$ $$
\phi_n(Q_2)=4(\b-\a-n-3)(\b-\a-n-2)(\b-\a-n-1)$$
$$\hskip 2cm +12(2\b-\a-n-3)(\b-\a-n-1)+3(3\b-\a-n-3).
$$
Setting $x=n+\a-\b+2$, the above equations become
$$
f(x)=\phi_n(Q_1)=3x(x-1)-4(x-\b)
$$
$$
g(x)=\phi_n(Q_2)=-4x(x-1)(x+1)+12(x-\b+1)(x-1)-3(x-2\b+1).
$$
Now we get $g(x)=q(x)f(x)+r(x)$, where
\begin{equation*}
q(x)=-\frac{4}{3}x+\frac{8}{9}\quad {\rm and}\quad
r(x)=-\frac{5}{9}(12 \b-13)x-15+\frac{130}{9}\b.
\end{equation*}
It is easy to see that $12 \b-13\ne0$. Then
$$
f(x)=\Bigg(-\frac{27x}{5(12 \b-13)}+\frac{18(3 \b -5)}{5(12
\b-13)^2}\Bigg)r(x) $$ $$+\frac{18(16 \b-15)(\b-1)(2\b-1)}{(12
\b-13)^2}.
$$
Thus to obtain that $\phi_n(Q_1)=\phi_n(Q_2)=0$, we must have
$\b=\a=\frac{1}{2}$ and $n=0$, or $\b =\frac{15}{16}$,
$\a=\frac{7}{16}$ and $n=0$. We have $V(\frac{1}{2},-\frac{1}{2})
\ot V'_{\a,\b}$ is simple if and only if $(\a,\b) \neq
(\frac{1}{2},\frac{1}{2})$ or $(\frac{7}{16},\frac{15}{16})$.
Moreover, $V(\frac{1}{2},-\frac{1}{2}) \ot
V'_{\frac{1}{2},\frac{1}{2}}$ has a unique simple submodule
$W_{\frac{1}{2},\frac{1}{2}}^{(0)}$ and
$\left(V(\frac{1}{2},-\frac{1}{2}) \ot
V'_{\frac{1}{2},\frac{1}{2}}\right)/W_{\frac{1}{2},\frac{1}{2}}^{(0)}$
is a highest weight module of highest weight $(\frac{1}{2},0)$;
$V(\frac{1}{2},-\frac{1}{2}) \ot V'_{\frac{7}{16},\frac{15}{16}}$
has a unique simple submodule $W_{\frac{7}{16},\frac{15}{16}}^{(0)}$
and $\left(V(\frac{1}{2},-\frac{1}{2}) \ot
V'_{\frac{7}{16},\frac{15}{16}}\right)/W_{\frac{7}{16},\frac{15}{16}}^{(0)}$
is a highest weight module of highest weight
$(\frac{1}{2},-\frac{1}{16})$.

(IV). Let $(c, h)=(\frac{1}{2}, 0)$.  In this case we take $p=3,
q=-4$ and $m=1$.  From \cite{FF} we see that $J(c,h)$ is  generated
by two singular vectors of degree $1$ and $6$.  It is not hard to
verify that $J(c,h)=U(\Vir_-)Q_1u+U(\Vir_-)Q_2 u$, where
$Q_1=d_{-1}$, $Q_2=64d_{-2}^3-93d_{-3}^2+264d_{-4}d_{-2}-108d_{-6}$
and $Q_1u$ is a singular vector in $M(\frac{1}{2},0)$ and $Q_2u$ is
obtained by subtracting some element in $U(\Vir_-)Q_1u$ from the
other singular vector.
Then we can get
$$
\phi_n(Q_1)=\b-\a-n-1,
$$
\begin{equation*}
\begin{split}
\phi_n(Q_2)=& 64(2\b-\a-n-6)(2\b-\a-n-4)(2\b-\a-n-2) \\
& \quad -93(3\b-\a-n-6)(3\b-\a-n-3) \\
& \quad +264(4\b-\a-n-6)(2\b-\a-n-2) \\
& \quad -108(6\b-\a-n-6).
\end{split}
\end{equation*}
Substituting $n=\b-\a-1$ in the second equation, we have
$$
\phi_n(Q_2)=2(16 \b-15)(\b-1)(2 \b-1).
$$
So letting $\phi_n(Q_1)=\phi_n(Q_2)=0$, we have $\b=\a=\frac{1}{2}$
and $n=-1$ or $\b =\a=\frac{15}{16}$ and $n=-1$.

We have $V(\frac{1}{2},0) \ot V'_{\a,\b}$ is simple if and only if
$(\a,\b) \neq (\frac{1}{2},\frac{1}{2})$ or
$(\frac{15}{16},\frac{15}{16})$. Moreover, $V(\frac{1}{2},0) \ot
V'_{\frac{1}{2},\frac{1}{2}}$ has a unique simple submodule
$W_{\frac{1}{2},\frac{1}{2}}^{(-1)}$ and $\left(V(\frac{1}{2},0) \ot
V'_{\frac{1}{2},\frac{1}{2}}\right)/W_{\frac{1}{2},\frac{1}{2}}^{(-1)}$
is a highest weight module of highest weight
$(\frac{1}{2},-\frac{1}{2})$; $V(\frac{1}{2},0) \ot
V'_{\frac{15}{16},\frac{15}{16}}$ has a unique simple submodule
$W_{\frac{15}{16},\frac{15}{16}}^{(-1)}$ and $\left(V(\frac{1}{2},0)
\ot V'_{\frac{15}{16},
\frac{15}{16}}\right)/W_{\frac{15}{16},\frac{15}{16}}^{(-1)}$ is a
highest weight module of highest weight
$(\frac{1}{2},-\frac{1}{16})$.

(V). Let $(c, h)=(\frac{1}{2}, -\frac{1}{16})$.  In this case we
take $p=3, q=-4$ and $m=2$.  From \cite{FF} we see that $J(c,h)$ is
generated by two singular vectors of degree $2,4$.  It is not hard
to verify that $J(c,h)=U(\Vir_-)Q_1u+U(\Vir_-)Q_2 u$, where
$Q_1=4d^2_{-1}+3d_{-2}$ and
$Q_2=144d_{-1}^4+600d_{-2}d^2_{-1}+264d_{-3}d_{-1}+49d^2_{-2}+36d_{-4}$.
Then we can get
$$
\phi_n(Q_1)=4(\b-\a-n-2)(\b-\a-n-1)+3(2\b-\a-n-2)
$$
\begin{equation*}
\begin{split}
&\phi_n(Q_2)\\
=& 144(\b-\a-n-4)(\b-\a-n-3)(\b-\a-n-2)(\b-\a-n-1) \\
& \quad +600(2\b-\a-n-4)(\b-\a-n-2)(\b-\a-n-1) \\
& \quad +264(3\b-\a-n-4)(\b-\a-n-1)\\
& \quad +49(2\b-\a-n-4)(2\b-\a-n-2) \\
& \quad +36(4\b-\a-n-4).
\end{split}
\end{equation*}
By straightforward computations we deduce that
$$\phi_n(Q_2)=\phi_n(Q_1)q(n)-20(\a+n)(16\b-15)$$
for some polynomial $q(n)$ and
$$\phi_n(Q_1)=(4\a-8\b+9+4n)(\a+n)+2(2\b-1)(\b-1).$$

Letting $\phi_n(Q_1)=\phi_n(Q_2)=0$, we get that $\a=0$ or
$\b=\frac{15}{16}$. If $\a=0$, we deduce that $ \b=\frac{1}{2}$ and
$n=0$. If $\b=\frac{15}{16}$, we can deduce that $\a=\frac{1}{16}$
and $n=0$, or $\a=\frac{9}{16}$ and $n=-1$. We have
$V(\frac{1}{2},-\frac{1}{16}) \ot V'_{\a,\b}$ is simple if and only
if $(\a,\b) \neq (0,\frac{1}{2}) ,(\frac{1}{16}, \frac{15}{16})$ or
$(\frac{9}{16},\frac{15}{16})$.

Moreover, $V(\frac{1}{2},-\frac{1}{16}) \ot V'_{0,\frac{1}{2}}$ has
a unique simple submodule $W_{0,\frac{1}{2}}^{(0)}$ and
$\left(V(\frac{1}{2},-\frac{1}{16}) \ot V'_{0,\frac{1}{2}}\right
)/W_{0,\frac{1}{2}}^{(0)}$ is a highest weight module of highest
weight $(\frac{1}{2},-\frac{1}{16})$; $V(\frac{1}{2},-\frac{1}{16})
\ot V'_{\frac{1}{16},\frac{15}{16}}$ has a unique simple submodule
$W_{\frac{1}{16},\frac{15}{16}}$ and
$\left(V(\frac{1}{2},-\frac{1}{16}) \ot
V'_{\frac{1}{16},\frac{15}{16}}\right
)/W_{\frac{1}{16},\frac{15}{16}}$ is a highest weight module of
highest weight $(\frac{1}{2},0)$; $V(\frac{1}{2},-\frac{1}{16}) \ot
V'_{\frac{9}{16},\frac{15}{16}}$ has a unique simple submodule
$W_{\frac{9}{16},\frac{15}{16}}$ and
$\left(V(\frac{1}{2},-\frac{1}{16}) \ot
V'_{\frac{9}{16},\frac{15}{16}}\right
)/W_{\frac{9}{16},\frac{15}{16}}$ is a highest weight module of
highest weight $(\frac{1}{2},-\frac{1}{2})$.


%

\end{exa}

\begin{rem} Comparing the case (III, IV, V) in the above example,
we see that using the tensor product $V(\frac{1}{2},-\frac{1}{2})\ot
V'_{\a,\b}$ we can and only can get the highest weight modules with
highest weight $(\frac{1}{2},0)$ and $(\frac{1}{2},-\frac{1}{16})$,
using the tensor product $V(\frac{1}{2},0)\ot V'_{\a,\b}$ we can and
only can get the highest weight modules with highest weight
$(\frac{1}{2},-\frac{1}{2})$ and $(\frac{1}{2},-\frac{1}{16})$, and
using the tensor product $V(\frac{1}{2},-\frac{1}{16})\ot
V'_{\a,\b}$ we can and only can get the highest weight modules with
highest weight $(\frac{1}{2},-\frac{1}{2})$, $(\frac{1}{2},0)$ and
$(\frac{1}{2},-\frac{1}{16})$.

It is well known that in the physics literature that the conformal
field theory associated to $V(1/2,0)$ has three simple modules
$V(1/2,0)$, $V(1/2,-1/2)$ and $V(1/2,-1/16)$ and the fusion rules
among these modules are given by
\begin{equation*}\begin{split}
&V(1/2,-1/2)\times V(1/2,-1/2)=V(-1/2,0),\\
&V(1/2,-1/2)\times V(1/2,-1/16)=V(1/2,-1/16),\\
&V(1/2,-1/16)\times V(1/2,-1/16)=V(1/2,-1/2)+ V(1/2,0),
\end{split}\end{equation*}
and $V(1/2,0)$ is the identity in the fusion rule relations. (See
\cite{DMZ}). From (III, IV, V) of Example 12, there seems to be some
mysterious relations among these three modules.

\end{rem}

\begin{exa}
For any coprime $p, q \in \{2, 3, 4, \ldots\}$, set
\begin{equation*}
c_{p,q} = 1 - 6\frac{(p-q)^2}{pq}.
\end{equation*}
From Theorem A in \cite{A} we know that the maximal submodule
$J(c,0)$ of the Verma module $M(c,0)$ is generated by $d_{-1}u$ if
and only if $c\ne c_{p,q}$. In this example we assume that $c\ne
c_{p,q}$. Then the irreducible module $V(c,0)=M(c,0)/J(c,0)$ has a
basis
\begin{equation*}
d_{-l}^{n_l}\cdots d_{-3}^{n_3}d_{-2}^{n_2}u, \quad
n_2,n_3,\ldots,n_l \in \Z_+.
\end{equation*}
If $ \b-\a =k\in \Z$, by Theorem \ref{simplicity} we know that
$W=V(c,0) \otimes V'_{\a,\b}$ has a smallest simple submodule
$W^{(k-1)}$ which is generated by $u\otimes v_{i}, i \geq k$ and
$W/W^{(k-1)}$ is a highest weight module of highest weight
$(c,\b-1)$.

\begin{claim}
$W/W^{(k-1)}$ is isomorphic to the Verma module $M(c,\b-1)$.
\end{claim}
\begin{proof}
It is enough to show that  $x(u \otimes t^{k-1}) \notin W^{(k-1)}$
for all nonzero homogeneous element $x \in U(\Vir_-)$. Let's prove
this  by contradiction. Suppose $x(u \otimes t^{k-1}) \in W^{(k-1)}$
for some $x \in U(\Vir_-)_{-m}$ where $m\in\Z_+$.

Since  $d_{-1}u=0$ and $d_{-1}(u \otimes t^k)=0$, 
then
\begin{equation*}
W^{(k-1)} = \sum_{i,l_2,\ldots,l_s \in \Z_+} \C d_{-s}^{l_s} \cdots
d_{-3}^{l_3} d_{-2}^{l_2} (u \otimes t^{k+i}).
\end{equation*}
Now $x(u \otimes t^{k-1}) \in W^{(k-1)}$ has the following form:
\begin{equation*}
x(u \otimes t^{k-1}) = \sum_{(l_2,\ldots,l_s) \in I}
a_{l_2,\ldots,l_s} d_{-s}^{l_s} \cdots d_{-3}^{l_3} d_{-2}^{l_2} (u
\otimes t^{k+\big(\sum\limits_{i=2}^s il_i\big)-m-1})
\end{equation*}
where $I\subseteq (\Z_+)^{s-1}$ is a finite subset. Notice that
$\big(\sum\limits_{i=2}^s il_i\big)-m-1 \in \Z_+$ for all
$(l_2,\ldots,l_s) \in I$, that is, $\sum\limits_{i=2}^s il_i \geq
m+1$.

On the other hand, we have $x(u \otimes t^{k-1}) \in
\bigoplus\limits_{i=0}^m V(c,0)_{-i} \otimes t^{k-m-1}$. Let
$(l_2,\ldots,l_s)$ be the index in $I$ such that
$\sum\limits_{i=2}^s il_i$ is maximal, then we have
\begin{equation*}
(d_{-s}^{l_s} \cdots d_{-3}^{l_3} d_{-2}^{l_2}u) \otimes t^{k-m-1}
\in
 \bigoplus\limits_{i=0}^m V(c,0)_{-i} \otimes t^{k-m-1},
\end{equation*}
i.e., $\sum\limits_{i=2}^s il_i \leq m$ which is a contradiction.
\end{proof}
\end{exa}

We like to conclude our paper by pointing out that we are not  able
to determine the conditions for  the quotient module
$W^{(n-1)}/W^{(n)}$ in Remark 9 to be simple or to be a Verma module
when it is not trivial. Our examples show that some of them are
simple while some others are Verma modules.

\newpage

\begin{center}
\bf Acknowledgments
\end{center}
The research was carried out during the visit of the first and
second authors at Wilfrid Laurier University in 2012 and 2013. They
gratefully acknowledge the support (a short term grant) and
hospitality of Wilfrid Laurier University. The second author is
partially supported by NSF of China (Grant 11101380) and CSC of
China (Grant 2010841037); the last author is partially supported by
NSF of China (Grant 11271109) and NSERC. The authors want to thank
Prof. R. Lu for a lot of helpful discussions when they were
preparing the paper.

\vskip 10pt

\noindent H.C.: School of Mathematical Sciences, University of
Science and Technology of China, Hefei 230026, Anhui, P. R. China.\\
e-mail: {\tt hjchenmath\symbol{64}gmail.com}.

\noindent X.G.: Department of Mathematics, Zhengzhou University,
Zhengzhou 450001, Henan, P. R. China. e-mail: {\tt
guoxq\symbol{64}zzu.edu.cn}.

\noindent K.Z.: Department of Mathematics, Wilfrid Laurier
University, Waterloo, ON, Canada N2L 3C5; and College of Mathematics
and Information Science, Hebei Normal (Teachers) University,
Shijiazhuang 050016, Hebei, P. R. China. e-mail: {\tt
kzhao\symbol{64}wlu.ca}.

\end{document}